\documentclass[12pt]{amsart}
\usepackage{graphicx}

\newtheorem{theorem}{Theorem}[section]
\newtheorem{corollary}[theorem]{Corollary}
\newtheorem{lemma}[theorem]{Lemma}

\theoremstyle{definition}
\newtheorem{definition}[theorem]{Definition}
\newtheorem{remarks}[theorem]{Remark}
\newtheorem{examples}[theorem]{Example}

\numberwithin{equation}{section}

\def\Koniec{\hbox to\hsize{\hfil$\diamond$}}
\def\nt{\noindent}

\newcommand{\Z}{\mathbb{Z}}
\newcommand{\R}{\mathbb{R}}
\def\r{\varrho}

\def\san{\sum_{\alpha\in \Z^d, \alpha\neq 0}}

\begin{document}

\title{ Box-splines orthogonal
projections }

\author{M. Be\'ska, K. Dziedziul}

\maketitle

Abstract. { Let $P$ be orthogonal projection on B-splines of degree
$r-1$ with equally spaced knots.  Sweldens and Piessens proved
that $P(x^r)-x^r$ is Bernoulli polynomial. We generalize Sweldens
ans Piessens's  result for box-splines. It gives the opportunity to
define the seminorm of Sobolev space in terms of the asymptotic
formula for the error in orthogonal projection. In second part we
deal with similar problems in $BV(\R^d)$. It is a modification of
Boschariev's asymptotic formula  for the functions of bounded
variation. 41A15, 41A35, 41A60. Keywords: box spline, Bernoulli
spline, asymptotic formula, orthogonal projection, function of
bounded variation BV, Marcinkiewicz's average. }


\nt

\section{Introduction}

Let $W^k_p(\R^d)$ denote the Sobolev spaces, $1\leq p<\infty$ with
the norm
\[
\|f\|_{k,p}=
\sum_{|\beta|\leq k} \|D^\beta f\|_p,%
\]
 where
\[
D^\beta f={\partial^{|\beta|} f \over
\partial x_1^{\beta_1}\cdots \partial x_d^{\beta_d}},
\quad \beta=(\beta_1,\cdots ,\beta_d), \] \[
\beta!=\beta_1!\cdot...\cdot\beta_d! \quad |\beta|=\beta_1+\cdots
+\beta_d.
\]
and
\[
\|f\|_p
=
\left( \int_{\R^d} |f|^p \right)^{1/p}.
\]
Let $V=\{v_1,v_2,\cdots,v_n\}$ denote a set of not necessarily
distinct, non zero vectors in $\Z^d\setminus \{0\}$, such that
\[
{\rm span}\{V\}=\R^d.
\]
We call such set admissible. The  box spline denoted by
$B_V(\cdot)$ corresponding to  $V$  is defined by requiring that
\begin{equation}\label{1}
 \int_{\R^d} f(x)B_V(x)\, dx=\int_{[0,1]^n}f(Vu)\, du
 \end{equation}
 holds for any continuous function $f$ defined on $\R^d$, see reference
\cite{BHR}. As usual
\[
Vu=u_1v_1+\cdots u_nv_n.
\]
The Fourier transform is given by
\[
{\widehat f}(\xi)
=
\int_{\R^d} f(t)e^{-2\pi i \xi\cdot t}\, dt.
\]
Here and subsequently "$\cdot$" denotes the scalar product in
$\R^d$. From (\ref{1}) by simple calculation we get that
\begin{equation}\label{2}
{\widehat B_V}(x)
=
\prod_{v\in V} g(x\cdot v),
\end{equation}
where
\[
g(t)={ 1-e^{-2\pi i t} \over 2\pi i t}.
\]
  We denote by $\# V$ the cardinality of the set $V$.
For an admissible set $V$ let
\begin{equation}\label{3}
\r_{V}=\max\{\,r\,:\,
                    {\rm span}\{V\setminus W\}=\R^d\,
                    \,\hbox{for all}\quad {W\subset V},  \# W=r,\,\}.
\end{equation}
 This parameter determines the  smoothness of a box splines
\[
B_V(\cdot)\in C^{\r_V-1}(\R^d)\setminus C^{\r_V}(\R^d).
\]
Let us define
\[
S_{L^2}(hV)=\overline{\hbox{span}} \{B_{V}(\cdot/h-\alpha)\, :\,
\alpha\in Z^d\},
\]
where $h>0$ and  the closure is taken in $L^2(\R^d)$. The
orthogonal projection from $L^2(\R^d)$ onto $S_{L^2}(hV)$ is
denoted by $P_h$. Denoting by $\langle \cdot,\cdot\rangle$ the inner product in
$L^2(\R^d)$, the orthogonal projection onto $S_{L^2}(hV)$ can be
written by $(P=P_1)$:
\begin{equation}\label{4}
P_h=\sigma_h \circ P \circ \sigma_{1/h}, \end{equation}
 where
\[
\sigma_h f(x) =f(x/h).
\]
 A family $V\subset Z^d$ is unimodular if for all $W\subset V$ with $\# W=d$
 we have $|\det W|\leq 1$. Set
\[
[]^\beta(x)=x^\beta
\]
and
\[
\gamma \leq \beta \quad {\rm iff} \quad \gamma_j \leq \beta_j,
j=1,\ldots,d.
\]

\section{Box-spline orthogonal projections}

 Let us define
 \begin{equation}\label{5}
L_\beta(x)=P([]^\beta)(x)-x^\beta,\qquad x\in \R^d.
\end{equation}
Note that in the univariate case $L_\beta$ is a Bernoulli spline
for $|\beta|=\r_V+1$, see  \cite{SP2}. In \cite{Dz} it was proved
in a particular case that  $L_\beta$ is linear combination of
Bernoulli  splines. In this section we generalize this results, see Theorem \ref{bernoul}
below. Applying this result we  simplify the asymptotic formula for  orthogonal
projection calculated in \cite{BD5}. Our method in the case of $L^2(\R^d)$
implies Theorem 2.2 of \cite{BD3}.

We know that $L_\beta$ is a periodic piecewise polynomial  and
from Lemma 3.4 in \cite{Dz} we have:

\begin{lemma}
Let $|\beta|\leq  \r_V+1$. Then
\begin{equation}\label{6} L_\beta(x)=\left({1\over
2 \pi i}\right)^{|\beta|} \san D^\beta \widehat{B_V}(\alpha)e^{2
\pi i \alpha\cdot x}.
\end{equation}
The series converges in every point of continuity of $L_\beta$.
\end{lemma}

In fact the problem of the convergence appears only for  box
splines  with $\r_V=0$ and on the boundary of that box-splines. By
Theorem \ref{bernoul} and Remark 2.9 we write $L_\beta$ for $|\beta|=1$ as a
linear combination of a Bernoulli  spline $B^1=x-1/2$, where the
Fourier series of $B^1$ converges also in the point of
discontinuity of $B^1$ (i.e $x=0$) to zero.

Define a set  $\Lambda$,
\begin{equation}\label{8}
\Lambda=\{U\subset V: \# U=\r_V+1, {\rm span}\{V\setminus U\}\neq
\R^d \}.
\end{equation}
Let $U\in \Lambda$. If for all $v\in V\setminus U$
\[
v\cdot \alpha=0
\]
we will denote that $\alpha \bot (V\setminus U)$. Note that the
vectors from $V\setminus U$ span a hyperplane in $\R^d$ i.e.
$d-1$-dimensional subspace.
 From definition of the set $\Lambda$ we get that for all $\alpha\neq 0$
 such that $\alpha \bot (V\setminus U)$
\[
v\cdot \alpha\neq 0 \quad {\rm for\quad all} \quad v\in U.
\]

\begin{definition}
Let us define Bernoulli splines \cite{St} for $U\in \Lambda$ by
\begin{equation}\label{9}
B(V,U)(x)=\sum_{\alpha\neq 0, \alpha\bot (V\setminus U)}
\quad\prod_{v\in U} {1\over 2 \pi i\alpha\cdot v} e^{2 \pi i
\alpha\cdot x},
\end{equation}
where $\alpha\in \Z^d$ and $x\in \R^d$.
\end{definition}

\begin{lemma}
 Let $V$ be unimodular and let $\alpha\in \Z^d\setminus
\{0\}$. Let $D^\beta \widehat{B_V}(\alpha)\neq 0$ for given
$|\beta|=\r_V+1$. Then if $\# U_\alpha\geq \r_V+1$ where
\[
U_\alpha=\{v\in V: \alpha\cdot v\neq 0\}
\]
then $\#U_\alpha=\r_V+1$.
\end{lemma}

\begin{proof}
 Note that by (\ref{2})
\[
D^\beta \widehat{B_V}(x)=\sum_{\gamma\leq \beta}  {\beta \choose
\gamma} D^\gamma\prod_{v\in U_\alpha} g(x\cdot v) D^{\beta-\gamma}
\prod_{v\in V\setminus U_\alpha} g(x\cdot v)
\]
 Since $\# U_\alpha\geq |\beta|=\r_V+1$ then for $\gamma<\beta$
\begin{equation}\label{10}
D^\gamma \prod_{v\in U_\alpha} g(x\cdot v)\Big|_{x=\alpha}=0.
\end{equation}
since $v\cdot \alpha\neq 0$ for $v\in U_\alpha$ and $g(v\cdot
\alpha)=0$. Since $g(0)=1$, (\ref{10}) shows that
\begin{equation}\label{11}
 D^\beta
\widehat{B_V}(\alpha)= D^\beta \prod_{v\in U_\alpha} g(x\cdot
v)\Big|_{x=\alpha}\neq 0 .
\end{equation}
(\ref{11})  implies theorem.
\end{proof}

\begin{lemma}
 Let $V$ be unimodular and let $\alpha\in \Z^d\setminus
\{0\}$. Let $D^\beta \widehat{B_V}(\alpha)\neq 0$ for given
$|\beta|=\r_V+1$. Then $U_\alpha\in \Lambda$.
\end{lemma}

\begin{proof}
 Note that $U_\alpha\neq \emptyset$ since $V$ spans $R^d$. If
$\# U_\alpha\geq \r_V+1$ then from Lemma 1.4 we get that $\#
U_\alpha=\r_V+1$. Moreover $\alpha\bot (V\setminus U_\alpha)$,
hence $span\{V\setminus U_\alpha\}\neq R^d$, it follows that
$U_\alpha\in \Lambda$.

Let us assume that $\# U_\alpha< \r_V+1$. But $\alpha\bot
(V\setminus U_\alpha) $, hence $span\{V\setminus U_\alpha\}\neq
R^d$. This is a contradiction with definition of $\r_V$ see
(\ref{3}).
\end{proof}

\begin{theorem}\label{bernoul}
 Let $V$ be unimodular. Let $|\beta|=\r_V+1$. Then $L_\beta$ is a linear combination of
 Bernoulli  splines
  \begin{equation}\label{14}
L_\beta(x)= P([]^\beta)(x)-x^\beta=  \sum_{U\in \Lambda}
C(\beta,U) B(V,U)(x)
\end{equation}
where constants
\[
C(\beta,U)=D^\beta (\prod_{v\in U}(x\cdot v)).
\]

 \end{theorem}

\begin{proof}
By Lemma 2.1 and Lemma 2.5 we get
\[
L_\beta(x)= \left({1\over 2 \pi i}\right)^{\r_V+1} \sum_{U\in
\Lambda} \quad \sum_{\alpha\neq 0, \alpha\bot (V\setminus U) }
D^\beta \widehat{B_V}(\alpha) e^{2\pi i\alpha\cdot x}.
\]
By (\ref{11}) we get that
 \begin{equation}\label{12}
D^\beta \widehat{B_V}(\alpha)=\prod_{v\in U_\alpha} {1\over v\cdot
\alpha} D^\beta (\prod_{v\in U_\alpha}(x\cdot v))\Big|_{x=\alpha}.
\end{equation}
Note that \[ C(\beta,U_\alpha)=D^\beta (\prod_{v\in
U_\alpha}(x\cdot v))\Big|_{x=\alpha}= D^\beta (\prod_{v\in
U_\alpha}(x\cdot v))
 \]
  Consequently
\[
L_\beta(x)=  \sum_{U\in \Lambda} C(\beta,U) B(V,U)(x)
\]

\end{proof}

Let us recall the results form \cite{BD5} and \cite{Dz}.

\begin{theorem}
  Let $1\leq p< \infty$. Let $V$ be unimodular.
Let $f\in W_p^{\r_V+1} (\R^d)$.  Then
 \begin{equation}\label{7}
 \lim_{h\to 0^+} \Big\| {f-P_h
f\over h^{\r_V+1}} \Big\|_p^p =
\end{equation}
 \[
 =\int_{\R^d} \Big( \int_{[0,1]^d}
\Big|  \sum_{|\beta|=\r_V+1} {1\over \beta!} D^\beta f(t)
L_\beta(x) \Big|^p \, dx \Big) dt. \]
\end{theorem}

Now we want to examine the right part of (\ref{7}).

\begin{theorem}
 Let $V$ be unimodular then
\[
\sum_{|\beta|=\r_V+1} L_\beta(x) { D^\beta f(t)\over \beta!}
=
\sum_{U\in \Lambda} D_U f(t) B(V,U)(x),
\]
where
\[
D_U=\prod_{v\in U} D_v
\]
and $D_v$ is the directional derivative.
\end{theorem}

\begin{proof}
 From Lemma 2.1 and 2.5 we get
\[
\sum_{|\beta|=\r_V+1} L_\beta(x) { D^\beta f(t)\over \beta!}=
\]
\[
\sum_{|\beta|=\r_V+1} \left({1\over 2\pi i}\right)^{\r_V+1} \san
D^\beta \widehat{B_V}(\alpha) e^{2\pi i\alpha\cdot x} { D^\beta
f(t)\over \beta!}=
\]
\[
\left({1\over 2\pi i}\right)^{\r_V+1} \sum_{|\beta|=\r_V+1}
\sum_{U\in \Lambda} \quad \sum_{\alpha\neq 0, \alpha\bot
(V\setminus U) } D^\beta \widehat{B_V}(\alpha) e^{2\pi
i\alpha\cdot x} { D^\beta f(t)\over \beta!}.
\]
Note that the sets $V\setminus U$, where $U\in\Lambda$ are
disjoint. Tedious calculation shows that
 \begin{equation}\label{13}
\sum_{|\beta|=\r_V+1} D^\beta (\prod_{v\in U_\alpha}(x\cdot
v))\Big|_{x=\alpha} {D^\beta f(t)\over \beta!}= D_{U_\alpha} f(t).
\end{equation}
Consequently using (\ref{12}), (\ref{14}) and (\ref{13}) we get
the theorem.
\end{proof}

\begin{remarks}  Let $V$ be unimodular. Then
 the functions $B(V,U)$, $U\in \Lambda$ are orthogonal in $L^2([0,1]^d)$.
Moreover since all norms in a finite dimension space are equivalent
$"\asymp "$ we get
\[
\int_{R^d} \Big( \int_{[0,1]^d} \Big|  \sum_{|\beta|=\r_V+1}
{1\over \beta!} D^\beta f(t) L_\beta(x) \Big|^p \, dx \Big) dt
=
\]
\[
= \int_{R^d} \Big( \int_{[0,1]^d} \Big| \sum_{U\in \Lambda} D_U
f(t) B(V,U)(x)  \Big|^p \, dx \Big) dt \asymp
\]
\[
\sum_{U\in \Lambda}  \int_{\R^d} |D_U f(t)|^p dt  \int_{[0,1]^d}
|B(V,U)(x)|^p  dx.
\]
For $p=2$ we have equality and we obtain the Theorem 2.2
\cite{BD3}, com. \cite{DU}.

\end{remarks}
\begin{remarks}
 Note also that for all $U\in \Lambda$ there is a vector
$\alpha_U \in \Z^d\setminus \{0\}$ such that
\[
\{\alpha\in \Z^d\setminus \{0\}: \alpha\bot (V\setminus U) \}
=
\{ k\alpha_U: k\in \Z\setminus \{0\} \}.
\]
Thus \[ B(V,U)(x)=B^{\r_V+1}(\alpha_U\cdot x) \prod_{v\in U}
{1\over \alpha_U\cdot v},\] where $B^k$ is Bernoulli polynomial
\[
B^k(t)=\sum_{n\in \Z\setminus \{0\}} {e^{2\pi i nt}\over (2 \pi i
n)^k}.\] Consequently changing the variable we get
\[
 \int_{[0,1]^d}|B(V,U)(x)|^p  dx
 =\int_{[0,1]^d} |B^{\r_V+1}(\alpha_U\cdot x)|^p dx
 \left(\prod_{v\in U}{1\over \alpha_U\cdot v} \right)^p
 \]\[
 =\int_0^1 |B^{\r_V+1}(t)|^p dt
 \left(\prod_{v\in U}{1\over \alpha_U\cdot v} \right)^p.\]

\end{remarks}

\section{Boschariev Theorem}

In 1969 S.V. Boschkariev proved the asymptotic formula of the
coefficients of Haar expansion for the functions of bounded
variation in $<0,1>$. In fact he proved that if $f$ is absolutely
continuous and $a_k(f)$ are the coefficients of Haar expansion
then
\[
\lim_{n\to \infty} \sqrt{2^n}
\sum_{k=1}^{2^n}|a_{2^n+k}(f)|={1\over 4} \vee_0^1 f.
\]
The approximation of the functions with bounded variation are now
attracted many mathematicians \cite{M}, \cite{W}, \cite{CDDD}.
Form our point of view we are interested of asymptotic formula
between picture $f$ and it's digital image $Pf$ in $L^1$ norm on
$\R^d$, $d>1$ where $P$ is orthogonal projection corresponding to
 a box spline
\[
B_V(x)=\chi_{[0,1]^d},
\]
where by $\chi_A$ we denote the characteristic function of the set
$A$.  From Theorem 2.2. and 2.7 we know that for $f\in W_1^1(\R^d)$
 \begin{eqnarray}\label{3.15} \lefteqn{
\lim_{h\to 0^+} \int_{\R^d} \left|{f-P_hf\over h}\right|}\\
\nonumber &=&\int_{\R^d} dt \int_{[0,1)^d}dx \left|{\partial f
(t)\over
\partial x_1} B^1(x_1)+\cdots +{\partial f (t)\over
\partial x_d} B^1(x_d) \right|\asymp |f|_{1,1},
\end{eqnarray}
and
\[
 |f|_{1,1}=  \sum_{j=1}^k \|{\partial f \over
\partial x_j}\|_1\asymp
 \int_{\R^d}|D f|,
 \]
 $Df=({\partial f \over \partial x_1},\ldots,{\partial f \over \partial
 x_d})$.
 Consequently, by asymptotic formula we get the semi-norm of $W_1^1(\R^d)$.
The challenge is to obtain similar result for $BV(\R^d)$. 
 
 We will consider the asymptotic for $f=\chi_E$, the characteristic functions of the bounded
 open set $E$ with Lipschitz boundary.
 Let us  recall the notation.
A function $u\in L^1(\R^d)$ whose partial derivatives in the sense
of distributions are measures (Radon signed  measures) with finite
variation is called a {\em function with bounded variation} i.e.
\[
Du=(\mu_1,\mu_2,\ldots,\mu_d)
\]
and
\[
 |\mu(\R^d)|<\infty, \quad i=1,\ldots,d.
\]
The class of all such functions will be denoted by $BV(\R^d)$. For
details see E. Giusti also W. P. Ziemmer, \cite{G}, \cite{Z}). If
$u\in BV(\R^d)$ the total variation $\|Du\|$ may be regarded as a
measure, if  $g\geq 0$ and $g$ is continuous then
\[
 \begin{aligned}\label{15} 
\|Du\|(g)=\sup\Big\{ \int_{\R^d} (Div \Phi)u:&
\quad \Phi=(\phi_1,\ldots,\phi_d)\in C^1_0(\R^d,\R^d),\\&
|\Phi(x)|=\sqrt{\phi_1(x)^2 +\cdots+\phi_d(x)^2}\leq g(x). \Big \}
\end{aligned}
\]
 {\em Caccioppoli set} in $\R^d$ which are known also as {\em set E
of finite perimeter} are defined by $\chi_E\in BV(\R^d)$. If
$D\chi_E=(\mu_1,\mu_2,\ldots,\mu_d)$ note that
\[
\mu_i<<\|D\chi_E\|
\]
then there is Radon-Nikodym derivative of $D\chi_E$ with respect
to $\|D\chi_E\|$ i.e.
\[
\nu(x)=\nu(x,E)=-{dD\chi_E \over d\|D\chi_E\|},
\]
$\nu(x)$ is called {\em generalized exterior normal} to $E$ at
$x$. We need also the notation of {\em reduced boundary} of $E$:
$x\in
\partial^* E$ if

1) $\|D\chi_E\|(B(x,r))>0$ for all ball with arbitrary radius
$r>0$

2) if \[ \nu_r(x)=-{D\chi_E(B(x,r))\over \|D\chi_E\|(B(x,r))},
\]
then the limit $\nu(x)=\lim_{r\to 0} \nu_r(x)$ exists with
$|\nu(x)|=1$.
It is also known that for every Borel set $B\subset \partial^*E$, we have $\|D\chi_E(B)\|(B)=H^{d-1}(B)$,where
$H^{d-1}$ is $(d-1)$-dimensional Hausdorff measure.

Note that it makes sense to consider modified right side of
formula of (\ref{3.15}) and from Remark 2.8 and 2.9  it is
equivalent to $|f|_{BV}$ i.e.
\begin{equation}\label{3.16}
\int_{\partial^* E} H^{d-1}(dt) \int_{[0,1)^d}dx
\left|\nu(t)\cdot B(x) \right| \asymp H^{d-1}(\partial^* E)=|f|_{BV},
\end{equation}
where $B(x)=(B^1(x_1),\ldots,B^1(x_d))$.
 Unfortunately the limit of (\ref{3.15}) for $f=\chi_E$ may
not exist, see example below. Consequently we introduce
Marcinkiewicz average of the operators $P_h$ \cite{BD0}, which was
used to prove equivalent norm in Hardy spaces. We restricted our
considerations on $h=2^{-n}$ i.e. we will consider the
averaging-projections $P_{2^{-n}}$ with respect to dyadic cubes of
side length $2^{-n}$. To pose correctly problems let us introduce
for $\tau\in R^d$
\[
(P^\tau)_{2^{-n}}f(x)=P_{2^{-n}}f(\cdot-\tau)(x+\tau).
\]
We start with a simply motivation.  Let $H^1$ be $1$-dimensional Hausdorff
 measure.

\begin{examples}
Let $A=[a_1,b_1]\times  [a_2,b_2]$. Then for $f=\chi_A$ and a.e.
$\tau \in \R^2$ and $0\leq \theta \leq 1/4$ there is a sequence
$\{n_k\}$ such that
\[
\lim_{n_k\to \infty} \int_{R^2} 2^{n_k}\left|f-(P^\tau)_{2^{-n_k}}
f\right|=2\theta H^{1}(\partial A)=2\theta |f|_{BV},
\]
\end{examples}
The easy proof is left to reader.
\begin{corollary} If $f=\sum_{j=1}^n d_j \chi_{A_j}$ where $d_j\in \R$ and
$A_j$ are rectangles then for a.e. $\tau \in \R^2$
\[
\lim_{n\to \infty} \sup\int_{R^2} 2^{n}\left|f-(P^\tau)_{2^{-n}}
f\right|={1\over 2} |f|_{BV}.
\]
\end{corollary}
Let $K(x,2^k)$ denote  cube in center at $x$ of  the side length
$2^k$. By simple calculation we get:
\begin{lemma} Let $E\subset \R^d$ be measurable bounded set. Then 
\begin{equation}\label{3.17}
\int_{\R^d} 2^{n}\left|\chi_E-(P^\tau)_{2^{-n}} \chi_E\right|= 2\sum_{K\in
Q_n} (2^{-n})^{d-1}  {|K\cap(E+\tau)|\over |K|}
\left(1-{|K\cap(E+\tau)|\over |K|}\right),\end{equation} where
$Q_n$ is a collection of all dyadic cubes of side length $2^{-n}$.
\begin{equation}\label{3.18}
\int_{[0,1]^d}\int_{\R^d} 2^{n}\left|\chi_E-(P^\tau)_{2^{-n}}
\chi_E\right|=2\int_{\R^d}2^n M_n(x) dx,
\end{equation}  where
\[
M_n(\tau)={|K(\tau,2^{-n})\cap E|\over |K|}
\left(1-{|K(\tau,2^{-n})\cap E|\over |K|}\right)
\]
\end{lemma}

For given vector $v\in \R^d\setminus \{0\}$ let us define
\[
F(v)=\int_{[0,1]^d} G_v(u) du.
\]
To define function $G_v$ we need for $u\in \R^d$
\[
\Pi_v(u)=\{y\in \R^d: (y-u)\cdot v=0\},
\]
\[
\Pi^+_v(u)=\{y\in \R^d: (y-u)\cdot v>0\},
\]
\[
\Pi^-_v(u)=\{y\in \R^d: (y-u)\cdot v<0\}.
\]
Thus
\[
G_v(u)={|\Pi_v^+(u)\cap[0,1]^d||\Pi_v^-(u)\cap[0,1]^d|\over
H^{d-1}(\Pi_v(u)\cap[0,1]^d)}.
\]

\begin{theorem} Let $E$ be open bounded with Lipschitz boundary.
Then
\[
\lim_{n\to \infty} \int_{[0,1^d]}d\tau\int_{\R^d}
2^{n}\left|\chi_E-(P^\tau)_{2^{-n}} \chi_E\right|=\int_{\partial^* E}
F(\nu(x))H^{d-1}(dx),
\]
\end{theorem}

We postpone the proof of the theorem. Now we  want to present the
main ideas. We consider $d=2$. In the ergodic theory [\cite{CFS}
pp. 69] it is known that for any continuous function $f$ on
$2$-dimensional torus $Tor^2$ with Lebesgue's measure $dt$
\begin{equation}\label{3.19}
\lim_{t\to \infty} {1\over 2t} \int_{-t}^t f(T_\alpha^s x)
ds=\int_{Tor^2} f(t) dt
\end{equation}
uniformly for all $x\in Tor^2$, where $T_\alpha^s$ is a
one-parameter group of translations (line $l=l(\alpha,x)$)
\begin{equation}\label{3.20}
T_\alpha^s x=(x_1+\cos(\alpha)s)mod 1,(x_2+\sin(\alpha) s)mod 1)
\end{equation}
 and $\cos(\alpha),\sin(\alpha)$ are rationally independent i.e.
 $\tan(\alpha)$ is a irrational number. Note that $\alpha$ is an angle
 between that line $l$ and $OX$. Let
 \[
 \nu(\alpha)=[-\sin\alpha,\cos\alpha].
 \]
 Then $\nu(\alpha)\perp l$.
Note also that  $\tan(\alpha)$ is a irrational number  iff for all
$x$ the line $t \to T_\alpha^t x$ is dense in $Tor^2$ iff there is
$x$  that the line $t \to T_\alpha^t x$ is dense in $Tor^2$.

 We will calculate the asymptotic formula for
polygons $E$ whose sides $l_1(\alpha_1), l_2(\alpha_2),\ldots,
l_k(\alpha_k)$ are not parallel to axis and each line determined
by $l_j$ is dense in $Tor^2$. It is crucial  that the above
theorem (\ref{3.19}) is true also for the bounded functions
$G_{\nu(\alpha_j}$ on $Tor^2=[0,1)^2$ and continuous on $(0,1)^2$.

The function $F(\nu(\alpha))$ has following properties:
$F(\nu(\alpha))=F(\nu(\alpha+\pi/2))$ and if $\pi/4\leq \alpha\leq
\pi/2$ then $F(\nu(\alpha))=F(\nu(\pi/2-\alpha))$.

 Let us present the function $F$ for $\alpha \in
[\pi/4,\pi/2]$.

 \vskip 1cm
\begin{figure}[tbh]
\begin{center}
\includegraphics[width=2.5cm,height=3cm, trim=100 350 400 300]{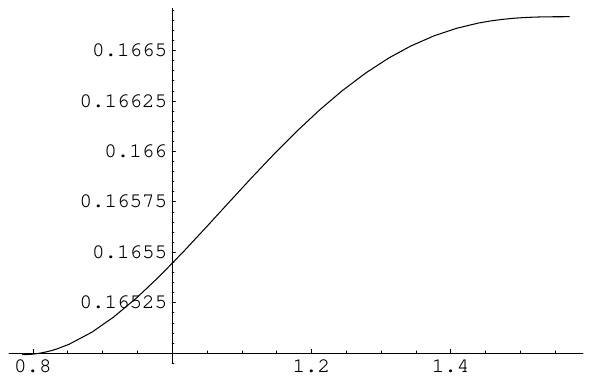}
\end{center}
\caption{}
\label{sineplot}
\end{figure}

 \vskip 0.5cm

\begin{theorem}
Let  $f=\chi_E$ where $E$ is a polygon of sides $l_1(\alpha_1),
l_2(\alpha_2),\ldots, l_k(\alpha_k)$ not parallel to axis and
$\sin\alpha_j$ and $\cos\alpha_j$ are rationally independent.
Uniformly for all $\tau \in \R^2$ we have
\[
\lim_{n\to \infty} \int_{\R^2} 2^{n}\left|f-(P_\tau)_{2^{-n}}
f\right|=2\sum_{j=1}^k |l_j| F(\nu(\alpha_j)).
\]
Hence
\[
\lim_{n\to \infty} \int_{[0,1]^2}d\tau \int_{\R^2}
2^{n}\left|f-(P_\tau)_{2^{-n}} f\right|=2\sum_{j=1}^k |l_j|
 F(\nu(\alpha_j)).
\]
\end{theorem}

{\sc Proof.} We need to calculate the limit of (\ref{3.17}),
(\ref{3.18}). It is immediate consequence of (\ref{3.19}).

\Koniec

It is interesting to compare the limit obtained in Theorem 3.5
with the left formula in (\ref{3.16}). Taking to account that $D
f=\sum_{j=1}^k(-\sin\alpha_j,\cos\alpha_j)d\mu_j$, where $\mu_j$
is a Lebesgue measure on $l_j$ we get
 \[
\int_{R^2}
d\mu(t) \int_{[0,1)^2}dx \left|\phi_1(t)  B^1(x_1)+ \phi_2(t) B^1(x_2)
\right|
\]
\[=\sum_{j=1}^k |l_j|
\int_{[0,1)^2}dx \left|-\sin\alpha_j B^1(x_1)+ \cos\alpha_j B^1(x_2)
\right|
\]
Let us compare
\[
 {
 \int_{[0,1)^2}dx \left|-\sin\alpha_j B^1(x_1)+ \cos\alpha_j B^1(x_2)\right|
 \over
 2F(\nu(\alpha_j))
 }
 \]
 for $\pi/4<\alpha_j< \pi/2. $

\vskip 1cm
\begin{figure}[tbh]
\begin{center}
\includegraphics[width=2.5cm,height=3cm,trim=100 350 400 300]{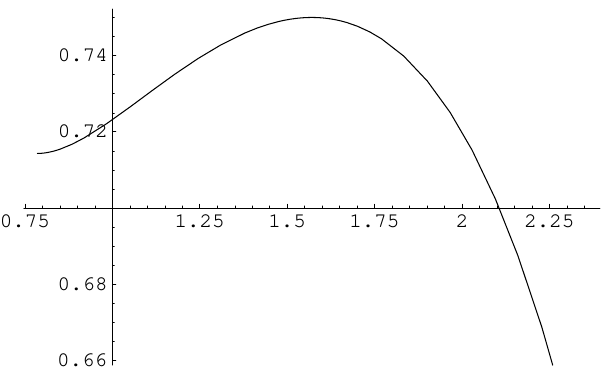}
\end{center}
\caption{}
\label{sineplot}
\end{figure}
\vskip 0.5cm

\begin{proof}[Proof of Theorem 3.4] Since $\partial E$ is Lipschitz we can
assume that for each $x\in\partial E$ there is local cartesian
coordinate such that

 there is a neighborhood  of $x$ $U_r$ and a Lipschitz function $\phi$ with a constant $L$
  such that
\[
U_r=B_r \times(-a,a),
\]
where $B_r$ is a ball  $B_r=\{x^*\in R^{d-1}, |x^*|<r \}$ and
$d>0$,
\[
\partial E \cap U_r=\{(x^*,\phi(x^*): x^*\in B_r\},
\]
\[
 E \cap U_r=\{x_d<\phi(x^*): x^*\in B_r, x_d\in(-a,a)\}.
\]
 To prove the theorem it is
sufficient to show that for any  $s<r$
\[
\lim_{n\to \infty} \int_{D\times (-a,a)}2^n M_n(x)
dx=\int_{\partial^* E\cap U} F(\nu(x))H^{d-1}(dx),
\]
where $D=B_s$,  $U=D\times (-a,a)$. Since $\phi$ is a Lipschitz
function for sufficient small $n$ (dependent on $s$)
\[
 \int_{D\times (-a,a)}2^n M_n(x)
dx=\int_D dx^*
\int_{\phi(x^*)-(L+1)2^{-n}}^{\phi(x^*)+(L+1)2^{-n}} 2^n
M_n((x^*,t)) dt
\]
Changing variables $t=\phi(x^*)+\tau2^{-n}$ we get
\[
=
\int_D dx^* \int_{-(L+1)}^{(L+1)} M_n((x^*,\phi(x^*+\tau2^{-n}))
d\tau
\]
From Rademacher theorem $\phi$ is differentiable at almost all
points. Let us denote a set of these points by $D^*\subset D$. Let
us fix $x^*\in D^*$. Note that $x=(x^*,\phi(x^*))\in \partial^*
E$. Let $K(x,2^k)$ denote  cube in center at $x$ of  the side
length $2^k$. Let
\[
T_{n,x^*}=\{t: K((x^*,t),2^{-n}) \cap \Pi_{\nu(x)}(x)\neq
\emptyset\}
\]
It is geometrically obvious that for bijection
$H(t)=(t-\phi(x^*))2^n$ there is $\delta=\delta(x^*)>0$ such that
$H(T_{n,x^*})=[-\delta,\delta]$. Moreover for all $\tau\in
[-\delta,\delta]$ and $i=+$ or $i=-$
\[
{|K((x^*,\phi(x^*)+\tau2^{-n}),2^{-n})\cap
\Pi^i_{\nu(x)}(x)|\over|K((x^*,\phi(x^*)+\tau2^{-n}),2^{-n})|}
\]
\[
=|K((x^*,\phi(x^*)+\tau),1)\cap\Pi^i_{\nu(x)}(x)|=:h^i_{x^*}(\tau).
\]
Note that
 \begin{eqnarray*} \lefteqn{
{|K((x^*,\phi(x^*)+\tau2^{-n}),2^{-n})\cap \Pi^i_{\nu(x)}(x)|\over
2^{-dn}} }\\ &=& {|K((x^*,\phi(x^*)+\tau2^{-n}),2^{-n})\cap
\Pi^i_{\nu(x)}(x)\cap E|\over 2^{-dn}}\\ &+&
{|K((x^*,\phi(x^*)+\tau2^{-n}),2^{-n})\cap \Pi^i_{\nu(x)}(x)\cap
(R^d \setminus E) |\over 2^{-dn}}.
\end{eqnarray*}
 From Theorem 5.6.5 \cite{Z} we infer that
\[
\lim_{n\to \infty} {|K((x^*,\phi(x^*)+\tau2^{-n}),2^{-n})\cap
\Pi^-_{\nu(x)}(x)\cap E
|\over|K((x^*,\phi(x^*)+\tau2^{-n}),2^{-n})|}=h^-_{x^*}(\tau).
\]
and
\[\lim_{n\to \infty} {|K((x^*,\phi(x^*)+\tau2^{-n}),2^{-n})\cap
\Pi^+_{\nu(x)}(x)\cap E
|\over|K((x^*,\phi(x^*)+\tau2^{-n}),2^{-n})|}=0.
\]
Consequently for all $\tau\in [-\delta,\delta]$
\[
\lim_{n\to \infty}
M_n(x^*,\phi(x^*+\tau2^{-n}))=h^-_{x^*}(\tau)h^+_{x^*}(\tau).
\]
and $\tau\not\in [-\delta,\delta]$
\[
\lim_{n\to \infty} M_n(x^*,\phi(x^*+\tau2^{-n}))=0.
\]

 From Lebesgue bounded convergence theorem
\[
\lim_{n\to\infty} \int_{-(L+1)}^{(L+1)} 2^n
M_n((x^*,\phi(x^*+\tau2^{-n})) d\tau= \int_{-\delta}^\delta
h^-_{x^*}(\tau)h^+_{x^*}(\tau)
\]
Making use of the last relation, we obtain
\[
\int_D dx^* \int_{-(L+1)}^{(L+1)} 2^n
M_n((x^*,\phi(x^*+\tau2^{-n})) d\tau=
\]
\[
=\int_{D^*} dx^* \int_{-\delta(x^*)}^{\delta(x^*)}
h^-_{x^*}(\tau)h^+_{x^*}(\tau) d\tau
\]
\[
=\int_{D^*} dx^*  \sqrt{1+|D\phi(x)|^2}{1\over
\sqrt{1+|D\phi(x)|^2}} \int_{-\delta(x^*)}^{\delta(x^*)}
h^-_{x^*}(\tau)h^+_{x^*}(\tau) d\tau
\]
\[
=\int_{D^*}\sqrt{1+|D\phi(x)|^2}F(\nu(x^*,\phi(x^*))) dx^*
\]
\[
=\int_{\partial^* E} F(\nu(x))H^{d-1}(dx).
\]
\end{proof}
By Theorem 5.7.3 \cite{Z} we get following corollary
\begin{corollary}
Let $E$ be a set of finite perimeter. Then
\[
\liminf_{n\to \infty} \int_{[0,1^d]}d\tau\int_{\R^d}
2^{n}\left|\chi_E-(P^\tau)_{2^{-n}} \chi_E\right|\geq 
2\int_{\partial^* E} F(\nu(x))H^{d-1}(dx).
\]
\end{corollary}
It seems to be true that
\[
\liminf_{n\to \infty} \int_{[0,1^d]}d\tau\int_{\R^d}
2^{n}\left|\chi_E-(P^\tau)_{2^{-n}} \chi_E\right|\asymp
H^{d-1}(\partial^* E).
\]

\bibliographystyle{amsplain}

\end{document}